\documentclass[10pt, letterpaper]{amsart}
\usepackage[utf8]{inputenc}
\usepackage{setspace}
\usepackage{mathrsfs}
\usepackage{amssymb}
\usepackage{latexsym}
\usepackage{amsfonts}
\usepackage{amsmath}
\usepackage{eucal}
\usepackage{bm}
\usepackage{bbm}
\usepackage{graphicx}
\usepackage[english]{varioref}
\usepackage[nice]{nicefrac}
\usepackage[all]{xy}
\usepackage{amsthm}
\usepackage{amssymb,amsthm,upref,amscd}

\def\op{\operatorname}

\def\mmod{\kern-1pt\operatorname{-mod}}

\newtheorem{theorem}{Theorem}[section]
\newtheorem{lemma}[theorem]{Lemma}
\newtheorem{definition}[theorem]{Definition}
\newtheorem{example}[theorem]{Example}
\newtheorem{remark}[theorem]{Remark}
\newtheorem{proposition}[theorem]{Proposition}
\newtheorem{corollary}[theorem]{Corollary}
\theoremstyle{proposition}

\numberwithin{equation}{section}

\begin{document}

\title[Category $\mathscr{X}$ ]{Complex representations of reductive algebraic groups with Frobenius maps in the category $\mathscr{X}$ }

\author{Junbin Dong}
\address{Institute of Mathematical Sciences, ShanghaiTech University, 393 Middle Huaxia Road, Pudong, Shanghai 201210, China.}
\email{dongjunbin@shanghaitech.edu.cn}


\date{September 5, 2023}


\begin{abstract}
In this paper we introduce an abelian  category  $\mathscr{X}(\bf G)$ to study the complex representations of a  reductive algebraic groups $\bf G$ with Frobenius map. We classify all the simple objects in $\mathscr{X}(\bf G)$ and show that this category  is a highest weight category.
\end{abstract}

\maketitle

\section{Introduction}

Let ${\bf G}$ be a connected reductive algebraic  group defined over the finite field $\mathbb{F}_q$ of $q$ elements. The rational representations of reductive  algebraic groups have been deeply studied and there have already been many fruitful results on this topic (see \cite{J}).
In this paper we consider the abstract representations of ${\bf G}$  over a field  $\Bbbk$.  When $\op{char}\Bbbk= \op{char} \bar{\mathbb{F}}_q$,
Borel and Tits determined all finite-dimensional abstract representations of ${\bf G}$ in \cite{BT}.  In the same paper, they showed that except the trivial representation, all other irreducible representations of $\Bbbk {\bf G}$ (the group algebra of $\bf G$) are infinite-dimensional if ${\bf G}$ is  semisimple and $\Bbbk $ is infinite with $\op{char}\Bbbk\neq \op{char} \bar{\mathbb{F}}_q$.  So studying the abstract representations in cross characteristic case is a very challenging problem since most of  the representations are infinite-dimensional.  Denote by $G_{q^a}$ the set of $\mathbb{F}_{q^a}$-points of $\bf G$, then we have ${\bf G}=\bigcup G_{q^a}$.  Xi made the  first attempt  to study the abstract representations of ${\bf G}$ over ${\bf \Bbbk}$ by taking the direct limit of the finite-dimensional representations of $G_{q^a}$ and he got many interesting results in \cite{Xi}. In particular, he showed that  the infinite-dimensional Steinberg module is irreducible when
$\op{char}\Bbbk=0$ or $\op{char}\Bbbk= \op{char} \bar{\mathbb{F}}_q$.  Later, Yang removed this restriction on $\op{char} \Bbbk$ and proved the irreducibility of the Steinberg module for any field $\Bbbk$  (see \cite{Yang}).  Recently, Putman and Snowden showed that  the Steinberg representation  is always irreducible for any infinite  Chevalley group (see \cite{PS}).
Let ${\bf B}$ be  a Borel subgroup of ${\bf G}$.  Motivated by Xi's idea in \cite{Xi},  the author and his collaborator studied the structure of the permutation module $\Bbbk [{\bf G}/{\bf B}]$  in  \cite{CD1} for the cross characteristic case and in \cite{CD2} for the defining characteristic case.
In particular, we get all the composition factors of $\Bbbk [{\bf G}/{\bf B}]$.

Let $\bf T$ be a maximal  torus contained in the  Borel subgroup $\bf B$ and $\theta$  be   a character of ${\bf T}$. Thus $\theta$  can also be regarded as a character of ${\bf B}$ by letting ${\bf U}$ (the unipotent radical of ${\bf B}$) act trivially.  The naive induced module $\mathbb{M}(\theta)=\Bbbk{\bf G}\otimes_{\Bbbk{\bf B}}{\Bbbk}_\theta$ is a natural object for study. When $\theta$ is trivial, this induced module is the permutation module  $\Bbbk [{\bf G}/{\bf B}]$.    We studied  the induced module $\mathbb{M}(\theta)$  for any field $\Bbbk$ with $\op{char}\Bbbk\neq \op{char} \bar{\mathbb{F}}_q$ or $\Bbbk=\bar{\mathbb{F}}_q$ in \cite{CD3}.  The  induced module $\mathbb{M}(\theta)$  has a composition series (of finite length) if $\op{char}\Bbbk\neq \op{char} \bar{\mathbb{F}}_q$. In the case $\Bbbk=\bar{\mathbb{F}}_q$ and $\theta$ is a rational character, $\mathbb{M}(\theta)$ has such composition series if and only if $\theta$ is antidominant. In both cases, the composition factors of $\mathbb{M}(\theta)$  are $E(\theta)_J$ with $J\subset I(\theta)$ (see Section 2 for the explicit definition). Therefore we have obtained a large class of irreducible modules of ${\bf G}$, almost of which are infinite-dimensional.

Now let $\Bbbk$ be an algebraically closed field of characteristic zero, e.g., $\Bbbk = \mathbb{C}$. Denote by  $\Bbbk{\bf G}$-Mod  the category of all  $\Bbbk {\bf G}$-modules. Obviously, this category is too large to study. On the other hand, since almost representations of ${\bf G}$ are infinite-dimensional, the category of finite-dimensional $\Bbbk {\bf G}$-modules is also not very good enough.
In the paper \cite{D1}, we introduce the principal representation category $\mathscr{O}({\bf G})$ which is defined to be the  full subcategory of $\Bbbk{\bf G}$-Mod whose object has composition factors $E(\theta)_J$'s. The category  $\mathscr{O}({\bf G})$   was supposed to have many good properties.  In particular, we conjectured that this category  is a highest weight category in the sense of \cite{CPS}. But soon after,  Chen constructed a counter example (see \cite{Chen}) to show that this conjecture is not true.  His example also tells us  that the  category $\mathscr{O}({\bf G})$ may be more complicated than we thought even for ${\bf G}= SL_2(\overline{\mathbb{F}}_q)$.

We may explore other good category  instead of  the category $\mathscr{O}({\bf G})$. The good category  should  contain the irreducible modules $E(\theta)_J$'s and many other abstract representations we have obtained. Note that  almost all the abstract representations of ${\bf G}$ are infinite-dimensional and the category $\Bbbk{\bf G}$-Mod is not semisimple.  Thus we hope the  good category  we are looking for  satisfies certain good properties such as ``finite-ness" and ``semi-simplicity", which is like  the BGG category $\mathscr{O}$ in the representations of complex semisimple Lie algebras.  In \cite{D2},  we have constructed  and studied a category which satisfies these conditions for ${\bf G}= SL_2(\overline{\mathbb{F}}_q)$. For general reductive algebraic group $\bf G$,  the category  $\mathscr{X}(\bf G)$ introduced in this paper  meets our expectations  and  has many good properties.  In particular,   $\mathscr{X}(\bf G)$  is a highest weight category in the sense of \cite{CPS}.

This paper is organized as follows: Section 2 contains some basic definitions, notations  and some known results about some special modules, such as  $\mathbb{M}(\theta)$,  $E(\theta)_J$ and $\nabla(\theta)_J$.  In Section 3, we introduce a new version of  Lie-Kolchin theorem for the study of  abstract representations of reductive algebraic groups.  In Section 4,  we introduce the category  $\mathscr{X}(\bf G)$  and study its basic properties.  Section 5 is devoted to show that $\mathscr{X}(\bf G)$  is an abelian category. We also classify all the simple objects in $\mathscr{X}(\bf G)$ in this section.  In Section 6 we prove that $\mathscr{X}(\bf G)$ has enough injective objects and  is a highest weight category. Moreover, we study the structure of certain finite-dimensional quasi-hereditary algebra  $\mathscr{A}_n$ whose representations are  closely related to the category $\mathscr{X}(\bf G)$.

\section{Preliminaries}

As in the introduction, let ${\bf G}$ be a connected reductive algebraic group defined over $\mathbb{F}_q$ with the standard Frobenius morphism $F$ induced by the automorphism $x\mapsto x^q$ on $\bar{\mathbb{F}}_q$. Let ${\bf B}$ be an $F$-stable Borel subgroup, and ${\bf T}$ be an $F$-stable maximal torus contained in ${\bf B}$, and ${\bf U}=R_u({\bf B})$ be the ($F$-stable) unipotent radical of ${\bf B}$.
We identify ${\bf G}$ with ${\bf G}(\bar{\mathbb{F}}_q)$ and do likewise for the various subgroups of ${\bf G}$ such as ${\bf B}, {\bf T}, {\bf U}$ $\cdots$.  For an algebraic group ${\bf H}$ defined over $\mathbb{F}_q$ (such as  ${\bf G}, {\bf B},  {\bf T}, {\bf U}$),  we denote by $H_{q^a}$ its $\mathbb{F}_{q^a}$-points.  We denote by $\Phi=\Phi({\bf G};{\bf T})$ the  root system of ${\bf G}$, and by $\Phi^+$ (resp. $\Phi^-$) the set of positive (resp. negative) roots determined by ${\bf B}$. Let $W=N_{\bf G}({\bf T})/{\bf T}$ be the corresponding Weyl group. For each $w \in W$, we denote by $\dot{w}$  one representative  in $N_{\bf G}({\bf T})$. Let  $\Delta=\{\alpha_i\mid i\in I\}$ be  the set of simple roots and $S=\{s_{\alpha_i}\mid i\in I\}$ be the corresponding simple reflections in $W$. We often identify $\Delta$, $S$ with $I$ if there is no confusion.  For each $J\subset I$, let $\Phi_J$ be the corresponding root system and ${\bf P}_J$ (resp. $W_J$)  be the standard parabolic subgroup of ${\bf G}$ (resp. $W$). Let $w_J$ be the longest element in $W_J$.

For each $\alpha\in\Phi$, let ${\bf U}_\alpha$ be the root subgroup corresponding to $\alpha$ and we fix an isomorphism $\varepsilon_\alpha: \bar{\mathbb{F}}_q\rightarrow{\bf U}_\alpha$ such that $t\varepsilon_\alpha(c)t^{-1}=\varepsilon_\alpha(\alpha(t)c)$ for any $t\in{\bf T}$ and $c\in\bar{\mathbb{F}}_q$. For any $w\in W$, we set
 $$\Phi_w^-=\{\alpha \in \Phi^+ \mid w(\alpha)\in \Phi^- \}, \ \ \Phi_w^+=\{\alpha \in \Phi^+ \mid w(\alpha)\in \Phi^+ \}.$$
Let ${\bf U}_w$ (resp. ${\bf U}_w'$) be the subgroup of ${\bf U}$ generated by all ${\bf U}_\alpha$  with $\alpha \in\Phi_w^-$ (resp. $\alpha\in\Phi_w^+$). The following properties are well known (see \cite{Car}):

\noindent(a) For  any root  $\alpha\in \Phi$,  we have $\dot{w}{\bf U}_\alpha \dot{w}^{-1}={\bf U}_{w(\alpha)}$;

\noindent(b) $ {\bf U}_w$ and ${\bf U}'_w$ are subgroups of  ${\bf U}$, and  $\dot{w}{\bf U}'_w\dot{w}^{-1} \subset {\bf U}$;

\noindent(c) The multiplication map ${\bf U}_w\times{\bf U}_w'\rightarrow{\bf U}$ is a bijection;

\noindent(d)  Let $\Phi^+= \{\gamma_1,\gamma_2,\dots,\gamma_n\}$. Then we have
 ${\bf U}= {\bf U}_{\gamma_1}{\bf U}_{\gamma_2}\dots {\bf U}_{\gamma_n}$.
 Each element  $u \in {\bf U}$ is uniquely expressible in the form
$u=u_{1}u_{2}\dots u_{n}$ with $u_{i}\in  {\bf U}_{\gamma_i}$;

\noindent(e) ({\it Commutator relations}) Given two positive roots $\alpha$ and $\beta$,  there exist a total ordering on $\Phi^+$ and integers $c^{mn}_{\alpha \beta}$ such that
$$[\varepsilon_\alpha(a),\varepsilon_\beta(b)]:=\varepsilon_\alpha(a)\varepsilon_\beta(b)\varepsilon_\alpha(a)^{-1}\varepsilon_\beta(b)^{-1}=
\underset{m,n>0}{\prod} \varepsilon_{m\alpha+n\beta}(c^{mn}_{\alpha \beta}a^mb^n)$$
for all $a,b\in \bar{\mathbb{F}}_q$, where the product is over all
integers $m,n>0$ such that $m\alpha+n\beta \in \Phi^{+}$, taken
according to the chosen ordering.

For different  chosen ordering of  $\Phi^+$,  the expression of an element $u\in {\bf U}$ in (d) may be different. However for any  simple root $\alpha$, its ${\bf U}_{\alpha}$-component is always the same according to (e).  Thus we denote by $\omega_{\alpha}(u)$ the  ${\bf U}_{\alpha}$-component of $u$ for each simple root $\alpha$. Let $X$ be a subset of $\Phi^+$ which makes  ${\bf U}_X$  a subgroup of ${\bf U}$,
then $X$ is said to be an enclosed subset of $\Phi^+$.

Let $\Bbbk$ be an algebraically closed field of characteristic zero.  In this paper, we consider the abstract representations of  ${\bf G}$ over $\Bbbk$. For any finite subset $X$ of ${\bf G}$, let $\underline{X}:=\displaystyle \sum_{x\in X}x \in \Bbbk {\bf G}$.  This notation will be frequently used in later discussion.

 Let $\widehat{\bf T}$ be the character group of ${\bf T}$.   Each $\theta\in\widehat{\bf T}$ can be regarded as a character of ${\bf B}$ by letting ${\bf U}$ act trivially.
Denote by  ${\Bbbk}_\theta$ the corresponding ${\bf B}$-module and we consider the induced module $\mathbb{M}(\theta)=\Bbbk{\bf G}\otimes_{\Bbbk{\bf B}}{\Bbbk}_\theta$. Let ${\bf 1}_{\theta}$ be a fixed nonzero element in ${\Bbbk}_\theta$.  For $x\in {\bf G}$, we abbreviate $x{\bf 1}_{\theta}:=x\otimes{\bf 1}_{\theta}\in\mathbb{M}(\theta)$ for simplicity. It is not difficult to see that $\mathbb{M}(\theta)$ has a basis $\{u \dot{w} {\bf 1}_{\theta}\mid w\in W,  u\in {\bf U}_{w^{-1}}\}$ by the Bruhat decomposition, where $\dot{w}$ is a fixed representative of $w \in W$.
According to \cite[Proposition 2.2]{CD3}, $\mathbb{M}(\theta)$ is indecomposable.

For each $i \in I$, let ${\bf G}_i$ be the subgroup of $\bf G$ generated by ${\bf U}_{\alpha_i}, {\bf U}_{-\alpha_i}$ and set ${\bf T}_i= {\bf T}\cap {\bf G}_i$. Then there is a natural homomorphism $\varphi_i: SL_2(\bar{\mathbb{F}}_q) \rightarrow {\bf G}_i$ such that
$$\varphi_i\left(\begin{array}{cc}1 &\ c\\0 &\ 1\end{array}\right)=\varepsilon_{\alpha_i}(c),\quad \varphi_i\left(\begin{array}{cc}1 &\ 0\\c &\ 1\end{array}\right)=\varepsilon_{-\alpha_i}(c)$$
for any $c\in\bar{\mathbb{F}}_q$.  For $\theta\in\widehat{\bf T}$, define the subset $I(\theta)$ of $I$ by $$I(\theta)=\{i\in I \mid \theta| _{{\bf T}_i} \ \text {is trivial}\}.$$
The Weyl group $W$ acts naturally on $\widehat{\bf T}$ by
$$(w\cdot \theta ) (t):=\theta^w(t)=\theta(\dot{w}^{-1}t\dot{w}), \quad \forall \theta\in \widehat{\bf T}.$$

For $J\subset I(\theta)$, we let ${\bf G}_J$ be the subgroup of $\bf G$ generated by ${\bf G}_i$, $i\in J$. Choose a representative $\dot{w}\in {\bf G}_J$ for each $w\in W_J$, and then  the element $w{\bf 1}_\theta:=\dot{w}{\bf 1}_\theta$ is well-defined. For $J\subset I(\theta)$, we set
$$\eta(\theta)_J=\sum_{w\in W_J}(-1)^{\ell(w)}w{\bf 1}_{\theta},$$
where $\ell(w)$ is the length of  $w\in W$.  Let $\mathbb{M}(\theta)_J=\Bbbk{\bf G}\eta(\theta)_J$ be the $\Bbbk {\bf G}$-module which is generated by $\eta(\theta)_J$. For $w\in W$, let  $\mathscr{R}(w)=\{i\in I\mid ws_i<w\}$.  For any subset $J\subset I$, we set
$$
\aligned
X_J &\ =\{x\in W\mid x~\op{has~minimal~length~in}~xW_J\}.
\endaligned
$$
We have the following proposition.

\begin{proposition}\cite[Proposition 2.5]{CD3} \label{MJ=KUW}
For any $J\subset I(\theta)$, the $\Bbbk {\bf G}$-module $\mathbb{M}(\theta)_J$ has the form
\begin{align} \label{MJ}
\mathbb{M}(\theta)_J=\sum_{w\in X_J}\Bbbk{\bf U}\dot{w}\eta(\theta)_J=\sum_{w\in X_J}\Bbbk{\bf U}_{w_Jw^{-1}}\dot{w}\eta(\theta)_J,
\end{align}
and the set $\{u\dot{w}\eta(\theta)_J \mid w\in X_J, u\in {\bf U}_{w_Jw^{-1}} \}$ is a basis of $\mathbb{M}(\theta)_J$.
\end{proposition}

For $J\subset I(\theta)$, we define
$$E(\theta)_J=\mathbb{M}(\theta)_J/\mathbb{M}(\theta)_J',$$
where $\mathbb{M}(\theta)_J'$ is the sum of all $\mathbb{M}(\theta)_K$ with $J\subsetneq K\subset I(\theta)$. We denote by $C(\theta)_J$ the image of $\eta(\theta)_J$ in $E(\theta)_J$.
Set
$$
\aligned
Z_J(\theta)&\ =\{w\in X_J \mid \mathscr{R}(ww_J)\subset J\cup (I\backslash I(\theta))\}.
\endaligned
$$

\begin{proposition} \cite[Proposition 2.7]{CD3} \label{DesEJ}
For $J\subset I(\theta)$, we have
\begin{align} \label{EJ}
E(\theta)_J=\sum_{w\in Z_J(\theta)}\Bbbk {\bf U}_{w_Jw^{-1}}\dot{w}C(\theta)_J,
\end{align}
and  the set $\{u\dot{w}C(\theta)_J \mid w\in Z_J(\theta), u\in {\bf U}_{w_Jw^{-1}} \}$ is a basis of $E(\theta)_J$.
\end{proposition}

The $\Bbbk {\bf G}$-modules $E(\theta)_J$ are  irreducible and thus we get all the composition factors of $\mathbb{M}(\theta)$ (see  \cite[Theorem 3.1]{CD3}). According to \cite[Proposition 2.8]{CD3}, one has that  $E(\theta_1)_{K_1}$ is isomorphic to $E(\theta_2)_{K_2}$ as $\Bbbk {\bf G}$-modules if and only if $\theta_1=\theta_2$ and $K_1=K_2$. An important and special case is that $\theta$ is trivial and $J=I$, and then we have  the Steinberg module $\op{St}$ whose irreducibility was proved in \cite[Theorem 3.2]{Xi}  and \cite[Theorem 2.2]{Yang}.

The irreducible $\Bbbk {\bf G}$-modules $E(\theta)_J$ can also be realized by parabolic induction.   For $J\subset I(\theta)$, set $J'=I(\theta)\backslash J$ and  let $$\nabla(\theta)_J= \mathbb{M}(\theta, J')=\Bbbk{\bf G}\otimes_{\Bbbk{\bf P}_{J'}}\Bbbk_\theta.$$  Let $E(\theta)_J'$ be the submodule of $\nabla(\theta)_J$ generated by $$D(\theta)_J:=\sum_{w\in W_J}(-1)^{\ell(w)}\dot{w}{\bf 1}_{\theta, J'}.$$
We see that $E(\theta)_J'$ is isomorphic to $E(\theta)_J$ as $\Bbbk {\bf G}$-modules by \cite[Proposition 1.9]{CD3}. Therefore $E(\theta)_J$ can be regarded as the socle of  $\nabla(\theta)_J$. Moreover it is easy to  have the following properties of $\nabla(\theta)_J$.

\begin{proposition} \label{nablaProp}
For any $J,K,L \subset I(\theta)$, we have

\noindent $(1)$ $\nabla(\theta)_J$ projects onto $\nabla(\theta)_K $ iff $J\supset K$.

\noindent $(2)$
$[\nabla(\theta)_J: E(\theta)_K]=\left\{
\begin{array}{ll}
 1 \ \  \text{if} \ J\supset K, \\
0\ \  otherwise.
\end{array}\right.
$

\noindent $(3)$ $\nabla(\theta)_J$ has simple socle $E(\theta)_J$ and simple head $E(\theta)_{\emptyset}$.

\noindent $(4)$  If $J\supset  K \supset L$, then $E(\theta)_K$ appears before $E(\theta)_L$ in any composition series of $\nabla(\theta)_J$.
\end{proposition}

\section{Lie-Kolchin Theorem}

 It is well known that all finite-dimensional rational
irreducible representations arise by inducing one-dimensional representations of
the Borel subgroup.  The key point is the Lie-Kolchin theorem,  which implies that every rational  irreducible representation of the Borel subgroup ${\bf B}$ is one-dimensional. This property is still true when we consider the finite-dimensional abstract representations of ${\bf B}$. It can be regarded as a version of the  Lie-Kolchin Theorem in the abstract representations of infinite Chevalley groups.

According to \cite[Lemma 1.4]{Xi}, each irreducible representation of ${\bf T} \subseteq {\bf B}$ is one-dimensional.  Noting that   ${\bf B} = {\bf U} \rtimes {\bf T}$, our version of Lie-Kolchin Theorem is as follows:

\begin{theorem} \label{LieKolchin}
Let ${\bf B}$ be a Borel subgroup of ${\bf G}$, then any  finite-dimensional irreducible  representation of ${\bf B}$ is one-dimensional, which affords a character on  $\bf T$.
\end{theorem}

\begin{proof}
  Let $M$ be a finite-dimensional irreducible  representation of ${\bf B}$.  Since $M$ is  finite-dimensional, we let $\lambda$ be a one-dimensional submodule of $M|_{\bf T}$. Using Frobenius reciprocity, we get
  $$\text{Hom}_{\bf T}(\lambda,  M|_{\bf T})\cong  \text{Hom}_{\bf B}(\text{Ind}^{\bf B}_{\bf T}\lambda,  M).$$
  In the following we will show  that  $\lambda$ (which is trivial on ${\bf U}$) is the unique simple quotient of $\text{Ind}^{\bf B}_{\bf T}\lambda$ as $\Bbbk {\bf B}$-modules. It is enough to show that
  $$L=\{\sum_{u\in{\bf U}} f_u u {\bf 1}_{\lambda} \mid \sum_{u\in{\bf U}} f_u =0, f_u\in \Bbbk \}$$
is the unique maximal submodule of $\text{Ind}^{\bf B}_{\bf T}\lambda$. Now let $\xi=\displaystyle \sum_{u\in{\bf U}} f_u u {\bf 1}_{\lambda}  \notin L$. Since the sum is finite, there exists an integer $a\in \mathbb{N}$ such that $u\in U_{q^a}$ whenever $f_u\ne 0$. Hence we have
$$\underline{U_{q^a}} \xi= (\sum_{u\in U_{q^a}} f_u)  \underline{U_{q^a}}  {\bf 1}_{\lambda} \in  \Bbbk {\bf B} \xi,$$
which implies that $\underline{U_{q^a}}  {\bf 1}_{\lambda} \in  \Bbbk {\bf B} \xi$.

For $X\subset \Phi^+$ such that ${\bf U}_X$ is a subgroup of ${\bf U}$, we claim that ${\bf 1}_{\lambda} \in  \Bbbk {\bf B} \xi$ if $\underline{U_{X, q^a}}  {\bf 1}_{\lambda} \in  \Bbbk {\bf B} \xi$ for some $a\in \mathbb{N}$. We prove this claim by do induction on $|X|$.  This claim is obvious if $X$ is the empty set. Now let $X=\{\gamma_1, \gamma_2, \dots, \gamma_m\}$ such that $\text{ht}(\gamma_1) \leq  \text{ht}(\gamma_2) \leq \dots \leq \text{ht}(\gamma_m)$. Denote by $Y=\{\gamma_2, \cdots, \gamma_m\}$ and it is easy to see that ${\bf U}_{Y}$ is also a subgroup of ${\bf U}$.
Let $b \in  \mathbb{N}$ such that  $a\mid b$, and  we get
$$ \underline{U_{Y, q^b}}\  \underline{U_{X, q^a}}  {\bf 1}_{\lambda} =q^{a|Y|}  \underline{U_{\gamma_1, q^a}}  \ \underline{U_{Y, q^b}} {\bf 1}_{\lambda}  \in  \Bbbk {\bf B} \xi. $$
Let $\mathfrak{D}_{q^b}$ be a subset of  $ \displaystyle \bigcap_{i=1}^m\gamma_i^{-1}(\mathbb{F}^*_{q^b})$ such that $\gamma_1: \mathfrak{D}_{q^b} \rightarrow \mathbb{F}^*_{q^b}$ is a bijection. Then  for any $c\in \mathbb{F}^*_{q^a}$,  we have
$$\sum_{t\in \mathfrak{D}_{q^b}}\lambda(t)^{-1} t\varepsilon_{\gamma_1}(c)  {\bf 1}_{\lambda}  =\sum_{t\in \mathfrak{D}_{q^b}}t \varepsilon_{\gamma_1}(c) t^{-1} {\bf 1}_{\lambda}  = \underline{U^*_{\gamma_1, q^b}}  {\bf 1}_{\lambda}  $$
and hence get
$$\sum_{t\in \mathfrak{D}_{q^b}}\lambda(t)^{-1} t  \underline{U_{\gamma_1, q^a}} \  \underline{U_{Y, q^b}} {\bf 1}_{\lambda} =(q^a-1)\underline{U_{X, q^b}} {\bf 1}_{\lambda}+  (q^b-q^a)\underline{U_{Y, q^b}} {\bf 1}_{\lambda}. $$
However it is easy to see that $\underline{U_{X, q^b}} {\bf 1}_{\lambda}\in \Bbbk {\bf B} \xi$, and thus we get $\underline{U_{Y, q^b}} {\bf 1}_{\lambda} \in \Bbbk {\bf B} \xi $. Therefore the claim is proved by induction.  Finally we get   ${\bf 1}_{\lambda} \in  \Bbbk {\bf B} \xi$, which implies that  $L$ is the unique maximal submodule of $\text{Ind}^{\bf B}_{\bf T}\lambda$. The theorem is proved.

\end{proof}

According to the proof of this theorem, we also get the following lemma which is useful in later discussion.

\begin{lemma}  \label{keylemma}
Let $M$ be a $\Bbbk {\bf B}$-module and $\lambda$ be a character of ${\bf B}$ (affording a character on  $\bf T$). Assume  $\eta \in M$ such that
$b\eta=\lambda(b) \eta, \forall b\in B$.  Then for any fixed integer $a\in  \mathbb{N}$,  $\eta$ is contained in the submodule of $M$ generated by $\underline{U_{q^a} }\eta$.
\end{lemma}

Using the setting and properties in \cite[Section 2]{FS}, when $\Bbbk$ is an algebraically closed field of characteristic $0$, $\Bbbk {\bf T}$ and $\Bbbk {\bf B}$ are both  locally Wedderburn algebras. Thus by \cite[Lemma 3]{FS}, we see that any  one-dimensional representation of ${\bf T}$ (resp.${\bf B}$) is an injective $\Bbbk {\bf T}$ (resp. $\Bbbk {\bf B}$)-module. In particular, we get the following proposition.

\begin{proposition} \label{SemisimpleTBmodule}
Any finite-dimensional representation of $\Bbbk {\bf T}$ (resp. $\Bbbk {\bf B}$) is semisimple, which is a direct sum of characters of ${\bf T}$ (resp. ${\bf B}$ ).
\end{proposition}

However the infinite-dimensional representations of $\Bbbk {\bf T}$ and $\Bbbk {\bf B}$ are very complicated. We even do not know the decomposition of  $\text{Ind}^{\bf B}_{\bf T}\text{tr}$ except for ${\bf G}=SL_2(\bar{\mathbb{F}}_q)$.
According to  a result of Borel and Tits \cite[Theorem 10.3 and Corollary 10.4]{BT}, we know that except the trivial representation, all other irreducible representations of $\Bbbk {\bf G}$ are infinite-dimensional if ${\bf G}$ is a semisimple algebraic group over $\bar{\mathbb{F}}_q$ and $\Bbbk $ is infinite with $\op{char}\Bbbk\neq \op{char} \bar{\mathbb{F}}_q$.  Thus the category of all  the finite-dimensional representations of $\Bbbk {\bf G}$ is not good enough for research.  On the other hand, if we regard the irreducible $\Bbbk {\bf G}$-modules $E(\theta)_J$ as $\Bbbk {\bf T}$ (or $\Bbbk {\bf B}$)-modules, they are not semisimple in general.
So looking for a good category to study the abstract representations of ${\bf G}$ seems a difficult  and challenging problem.

\section{Category  $\mathscr{X}$ and its basic properties}

Let $M$ be a  $\Bbbk {\bf G}$-module. For  a character $\theta\in \widehat{\bf T}$,  let
$$ M_{\theta}= \{v\in M \mid  t v= \theta(t) v,\  \forall \ t\in {\bf T} \}.$$
A character $\theta \in \widehat{\bf T}$ is called a ${\bf T}$-eigenvalue of $M$ if $M_{\theta}\ne 0$ and then we denote the  ${\bf T}$-eigenvalues  of $M$ by $\text{Ev}(M)$.
We set $M^{\bf T}=\displaystyle  \bigoplus_{\theta\in \widehat{\bf T} } M_{\theta} $.
As $\Bbbk {\bf T}$-module,  $M^{\bf T}$ is semisimple  by Proposition
 \ref{SemisimpleTBmodule}. However, we should note that $M\ne M^{\bf T}$ in general.  For example,  one can look at the infinite-dimensional Steinberg module.

\begin{definition}
The category $\mathscr{X}({\bf G})$ is defined to be the full subcategory of $\Bbbk {\bf G}$-Mod whose objects are the modules
 satisfying the following  condition:

 \noindent ($\star$)  $M^{\bf T}$ is a finite-dimensional space and there exists a basis $\{\xi_1, \xi_2, \dots, \xi_m\}$ of  ${\bf T}$-eigenvectors such that $M=\displaystyle \bigoplus_{i=1}^m \Bbbk {\bf U} \xi_i$, and moreover for each $i=1,2,\dots, m$,  there exists an enclosed  subset $X_i \subset \Phi^+$ such that $\Bbbk {\bf U} \xi_i \cong \op{Ind}^{\bf U}_{{\bf U}_{X_i}}\op{tr}$ as $\Bbbk {\bf U}$-modules.

\end{definition}

We have studied the category  $\mathscr{X}({\bf G})$  for ${\bf G}=SL_2(\bar{\mathbb{F}}_q)$ in  \cite{D2}. For the convenience of later discussion, we summarize the main results here.

\begin{proposition} \label{SL2result}
For  ${\bf G}=SL_2(\bar{\mathbb{F}}_q)$, the category  $\mathscr{X}({\bf G})$ is an abelian category and a Krull–Schmidt category.  The set  of  indecomposable objects in $\mathscr{X}({\bf G})$  is
$$\op{Ind}({\bf G})= \{ \Bbbk_{\op{tr}}, \op{St}, \mathbb{M}(\theta)\mid \theta \in \widehat{{\bf T}}\}, $$
and  the set  of simple objects in  the category $\mathscr{X}(\bf G)$ is
$$\op{Irr}({\bf G})= \{ \Bbbk_{\op{tr}}, \op{St}, \mathbb{M}(\theta)\mid \theta \in \widehat{{\bf T}} \ \text{is nontrivial}\}.$$
\end{proposition}

The $\Bbbk {\bf G}$-modules $\mathbb{M}(\theta)_J$, $\nabla(\theta)_J$  and $E(\theta)_J$ introduced in Section 2 are in the category $\mathscr{X}({\bf G})$.  For  $M\in  \mathscr{X}({\bf G})$,  we emphasize that not each basis of  $M^{\bf T}$ satisfies the condition ($\star$). For example, we consider the case  ${\bf G}=SL_2(\bar{\mathbb{F}}_q)$ and $M=\mathbb{M}(\op{tr})$. Then we have
$\mathbb{M}(\op{tr})=\Bbbk {\bf 1}_{\op{tr}} \oplus  \Bbbk {\bf U} \dot{s}  {\bf 1}_{\op{tr}}  $
and thus $\{{\bf 1}_{\op{tr}} ,   \dot{s}  {\bf 1}_{\op{tr}}  \}$ is a basis of $\mathbb{M}(\op{tr})^{\bf T}$  which  satisfies the condition ($\star$). However the basis  $\{(1-\dot{s}){\bf 1}_{\op{tr}} ,  (1+ \dot{s})  {\bf 1}_{\op{tr}}  \}$  has the property that
 $$(x-y)(1-\dot{s}){\bf 1}_{\op{tr}}=(y-x)  (1+ \dot{s})  {\bf 1}_{\op{tr}},  \quad \forall x,y\in {\bf U}. $$
 Therefore this basis does not satisfy condition ($\star$).

\begin{definition}
For  $M \in \mathscr{X}({\bf G})$,  a basis $\{\xi_1, \xi_2, \cdots, \xi_m\}$ of $M^{\bf T}$ satisfying the   condition  ($\star$) is called a good basis of $M^{\bf T}$.
\end{definition}

For $\theta\in \widehat{\bf T}$ and  $\gamma\in \Phi^+$, we let
$$ M_{(\theta, \gamma)}= \{v\in M^{\bf T} \mid  \varepsilon_{\gamma}(x) v=v, \forall x\in  \bar{\mathbb{F}}_q \}.$$
For $X \subset \Phi^+$,  we set  $M_{(\theta, X)}= \bigcap_{\gamma\in X} M_{(\theta, \gamma)}$ and then let
$$M_{\{\theta, X\}}= M_{(\theta, X)} \big/ \sum_{Y \varsupsetneq X} M_{(\theta, Y)}.$$
We denote  by $p_{\theta, X}: M_{(\theta, X)} \rightarrow M_{\{\theta, X\}} $ the projection.

\begin{definition}
Let $M\in \mathscr{X}({\bf G})$ and $\xi \in M^{\bf T}$.   For $\theta\in \widehat{\bf T}$ and $X\subset \Phi^+$,  a pairing $\{\theta, X\}$ is called the weight of $\xi$ if  $\xi\in M_{(\theta, X)}$ and $p_{\theta, X}(\xi)$ is nonzero.
\end{definition}

\begin{remark} \normalfont
For $\theta\in \widehat{{\bf T}}$,  not every vector  in  $M_{\theta}$ has a weight. For example, when ${\bf G}=SL_3(\bar{\mathbb{F}}_q)$, we consider the induced module  $\mathbb{M}(\op{tr})$. Denote  the simple reflections by $\{s_1, s_2\}$,  and then it is easy to see that $s_1 {\bf 1}_{\op{tr}} +s_2  {\bf 1}_{\op{tr}}$ does not have a weight.
\end{remark}

Let $\{\xi_1, \xi_2, \cdots, \xi_m\}$ be a good basis of $M^{\bf T}$.
Suppose $\xi_i\in M_{\theta_i}$ and  $\Bbbk {\bf U} \xi_i \cong \text{Ind}^{\bf U}_{{\bf U}_{X_i}} \text{tr}$ as $\Bbbk {\bf U}$-modules, then  $\xi_i$  has the weight   $\{\theta_i, X_i\}$. The weights of  the  good basis $\{\xi_1, \xi_2, \cdots, \xi_m\}$  is denoted by  $\text{Wt}(M)$.  It is easy to see that $ \text{Wt}(M)$    is independent of the  choice of  good bases of $M^{\bf T}$.  On the other hand, we can directly see that
 $$\text{Wt}(M)=\{\{\theta, X\} \mid \dim M_{\{\theta, X\}}  \ne 0  \}. $$

\begin{definition}
The character of $M\in \mathscr{X}({\bf G}) $  is defined to be
  $$\text{ch} M=  \sum \dim M_{\{\theta, X\}} \{\theta, X\}.$$
\end{definition}

\begin{example}  \label{Example} \normalfont
(a) The weights of $\mathbb{M}(\theta)_J$ are $\{\theta^w, \Phi^+_{w_Jw^{-1}}\}$ with $w\in X_J$ and
$$\text{ch} \mathbb{M}(\theta)_J=\sum_{w\in X_J}  \{\theta^w, \Phi^+_{w_Jw^{-1}}\}. $$

\noindent (b)  The weights of $E(\theta)_J$ are $\{\theta^w, \Phi^+_{w_Jw^{-1}}\}$ with $w\in Z_J(\theta)$ and
$$\text{ch} E(\theta)_J=\sum_{w\in Z_J(\theta)}  \{\theta^w, \Phi^+_{w_Jw^{-1}}\}. $$

\noindent (c)  The weights of $\nabla(\theta)_J$ are $\{\theta^w, \Phi^+_{w_Jw^{-1}}\}$ with $w\in X_{J'}$, where $J'= I \setminus J$.
The character of $\nabla(\theta)_J$ is
$$\text{ch} \nabla(\theta)_J =\sum_{w\in X_{J'}}  \{\theta^w, \Phi^+_{w_Jw^{-1}}\}.$$

\end{example}

\begin{proposition} \label{expressionofbasis}
For $M\in \mathscr{X}({\bf G}) $, let $\{\xi_1, \xi_2, \dots, \xi_m\}$  be a  good bases of  $M^{\bf T}$ and assume that the weight of $\xi_i$ is $\{\theta_i, X_i\}$.   Suppose $\xi \in M^{\bf T}$ is of weight $\{\theta, X\}$ and we write $\xi=\displaystyle  \sum_{i=1}^m f_{i}\xi_i$.
If $f_i\ne 0$, then $\xi_i\in M_{\theta}$ and $X_i \supseteq X$.  Moreover,  there exists at least one integer $i_0$ such that $f_{i_0} \ne 0$ and $X_{i_0}= X$.
\end{proposition}

\begin{proof} The property  $\theta_i= \theta$ is obvious when $f_{i} \ne 0$. Suppose that there is an  integer  $j_0$ such that  $f_{j_0} \ne 0$ and $X_{j_0} \nsupseteqq X$, then there exists  a  positive root $\gamma \in X$ and $\gamma \notin X_{j_0}$.  Hence we get $\varepsilon_{\gamma}(z) \xi=\xi, \forall z\in \bar{\mathbb{F}}_q $.   Thus we have
$$ \displaystyle \sum_{i=1}^m f_{i} (\varepsilon_{\gamma}(z)- 1) \xi_i =0,  \quad  \forall z\in  \bar{\mathbb{F}}_q.$$
However noting that $ f_{j_0}(\varepsilon_{\gamma}(z)- 1) \xi_{j_0} \ne 0$ for $z\in \bar{\mathbb{F}}^*_q$,  we get  a contradiction since  $\{\xi_1, \xi_2, \dots, \xi_m\}$ is a good basis of  $M^{\bf T}$.

Recall  the definition of  $p_{\theta, X}: M_{(\theta, X)} \rightarrow M_{\{\theta, X\}} $.  Then we have  $p_{\theta, X}(\xi) \ne 0$. Now suppose $X_i \supsetneq X$ whenever   $f_{i} \ne 0$,  then $p_{\theta, X}(\xi_i) = 0$
and thus $p_{\theta, X}(\xi) = 0$,   which is a contradiction. So there exists at least an integer $i_0$ such that $f_{i_0} \ne 0$ and $X_{i_0}= X$. The proposition is proved.
\end{proof}

\begin{corollary} \label{twoGB}
For $M\in \mathscr{X}({\bf G}) $, let $B_1 $ and $B_2$  be two good bases of $M^{\bf T}$.
Set $B_{1,\theta}=B_1\cap M_{\theta}$ and  $B_{2,\theta}=B_2\cap M_{\theta}$ for each $\theta\in \widehat{\bf T}$.
Then the transition matrix between $B_{1,\theta}$ and $B_{2,\theta}$  is a blocked upper triangular matrix when the vectors of  $B_{i,\theta}$ ($i=1,2$) are in  an  appropriate order.
\end{corollary}

\begin{proposition} \label{constrgoodbasis}
  Let $M\in \mathscr{X}({\bf G}) $. For each $\{\theta, X\} \in \op{Wt}(M)$, let $B_{\{\theta, X\}}$ be a subset of $M^{\bf T}$ such that $p_{\theta, X}( B_{\{\theta, X\}} )$ is a basis of $M_{\{\theta, X\}}$ and $| B_{\{\theta, X\}} |=\dim M_{\{\theta, X\}}$.
Then $\displaystyle  \bigcup_{\{\theta, X\} \in \op{Wt}(M) } B_{\{\theta, X\}}$ is a good basis of $M^{\bf T}$.
\end{proposition}

\begin{proof}  Let $\{\xi_1, \xi_2, \dots, \xi_m\}$ be a good basis of $M_{\theta}$ and $\{\eta_1, \eta_2, \dots, \eta_m\}$ be a basis which is constructed in the proposition.  It is enough to show that $\displaystyle \sum_{i=1}^m  \Bbbk {\bf U} \eta_i$ is a direct sum and
$\displaystyle \sum_{i=1}^m  \Bbbk {\bf U} \eta_i= \bigoplus_{i=1}^m \Bbbk {\bf U} \xi_i.$
Without loss of generality, we can assume that these two sets are both bases of  $M_{\theta}$ for a fixed $\theta\in  \widehat{{\bf T}}$.
Suppose the weights of $M_{\theta}$ are
$$\{\theta,  X_1\}, \  \{\theta,  X_2\}, \  \cdots,\   \{\theta,  X_n\},$$
where $|X_1| \geq |X_2| \geq \cdots \geq  |X_n|$. Let $i_0, i_1, \dots, i_n$ be integers such that
$$0=i_0 < i_1 < i_2 < \dots < i_n=m.$$
We assume that  the weight of $\xi_j, \eta_j$ is $ \{\theta,  X_k\}$ when $i_{k-1}< j \leq i_{k}$.
For convenience,  set $Y_k= \Phi^+ \setminus X_k$ for $k=1,2, \dots, n$.
According to  the same discussion in the proof of Proposition \ref{expressionofbasis} and by induction on the integer $k$, it is not difficult to see  that
$\displaystyle \sum_{j\leq i_k}\Bbbk \xi_j = \sum_{j\leq i_k}\Bbbk \eta_j$
 for $k=1, 2, \dots, n$. Since  the weight of $\xi_j, \eta_j$ is the same, we get
 $\displaystyle \sum_{j\leq i_k}\Bbbk  {\bf U} \eta_j =\bigoplus_{j\leq i_k}\Bbbk {\bf U}\xi_j$.  In particular,
we have $\displaystyle \sum_{i=1}^m  \Bbbk {\bf U} \eta_i= \bigoplus_{i=1}^m \Bbbk {\bf U} \xi_i.$
 It is easy to see that the following set
 $$ \bigcup_{k=1}^n \bigcup_{ i_{k-1} <  j\leq i_k  } \{u\xi_j\mid  u\in {\bf U}_{Y_k}\}$$
 is a basis of $\displaystyle  \bigoplus_{i=1}^m \Bbbk {\bf U} \xi_i$.
 Suppose that  $\displaystyle \sum_{j=1}^m \sum_{u\in {\bf U}} f_{j,u} u\eta_j=0$, where $u\in {\bf U}_{Y_k}$ for $ i_{k-1} <  j\leq i_k$.
 Since  $\displaystyle \sum_{j=1}^m \Bbbk \xi_j = \sum_{j=1}^m \Bbbk \eta_j$, we let $\eta_j=\displaystyle  \sum_{l=1}^n a_{jl}\xi_l$.
 Then we get
$$\displaystyle \sum_{l=1}^m \sum_{u\in {\bf U}} ( \sum_{j=1}^m f_{j,u}a_{jl} )u\xi_l=0,$$
 which implies that $\displaystyle \sum_{j=1}^m f_{j,u}a_{jl} =0$ for any $l$ and $u$.
Noting that the matrix $A=(a_{ij})_{n\times n}$ is invertible, we get $f_{j,u}=0$ and thus  $\displaystyle \sum_{i=1}^m  \Bbbk {\bf U} \eta_i$ is a direct sum.  The proposition is proved.

\end{proof}

In the following we let $M \in  \mathscr{X}({\bf G}) $ and  fix a good basis $\{\xi_1, \xi_2, \cdots, \xi_m\}$ of $M^{\bf T}$.   Assume that  the weight of $\xi_i$ is $\{\theta_i, X_i\}$.  We claim that for each $k=1,2, \cdots, m$ and a simple root $\alpha\in \Delta$, we have
\begin{align} \label{suinu1}
\dot{s}_{\alpha} \varepsilon_{\alpha}(1) \xi_k \in \bigoplus_{i=1}^m \Bbbk {\bf U}_{\alpha} \xi_i.
\end{align}
Indeed, we can write
 \begin{align} \label{suinu2}
 \dot{s}_{\alpha} \varepsilon_{\alpha}(1) \xi_k= \sum_{i=1}^m\sum_{u\in {\bf U}} f_{i,u} u \xi_i, \quad f_{i,u}\in \Bbbk,
 \end{align}
and thus there exists an integer $a$ such that  $u\in U_{q^a}$  whenever $f_{i,u} \ne 0$.
Now suppose that there exists an element $u_0\in {\bf U}$ such that $f_{i,u_0} \ne 0$ and  $u_0\notin  {\bf U}_{\alpha}$.
In this case, we can chose $t\in \dot{s}_{\alpha} (\text{Ker}\alpha )\dot{s}^{-1}_{\alpha}$ and $t u_0t^{-1} \notin  U_{q^a}$. Noting that $t$ acts trivially on the left side of  (\ref{suinu2}) but not trivially on the right side of  (\ref{suinu2}),   we get a contradiction.  Thus we have proved the following proposition.

\begin{proposition} \label{SL2module}
Let $M\in  \mathscr{X}({\bf G})$.  Given  a simple root $\alpha\in \Delta$, let ${\bf G}_{\alpha}$ be the subgroup generated by ${\bf U}_{\alpha},  {\bf U}_{-\alpha}$. Then $M^{\alpha} : = \Bbbk {\bf U}_{\alpha} M^{\bf T} $ is a $\Bbbk {\bf G}_{\alpha}$-module.
\end{proposition}

It is well known that there exists a natural morphism from $SL_2(\bar{\mathbb{F}}_q)$ to $ {\bf G}_{\alpha}$.  Thus
$M^{\alpha}$ can be regarded as a  $SL_2(\bar{\mathbb{F}}_q)$-module.
 For $M\in \mathscr{X}({\bf G})$ and a simple root $\alpha\in \Delta$, we define a projection
$$\Gamma_{\alpha}:  M \rightarrow  M^{\alpha},  \quad  \sum f_{i,x} x\xi_{i} \mapsto \sum f_{i,x} \omega_{\alpha}(x) \xi_{i},$$
where  $\omega_{\alpha}(x)$ is  the  ${\bf U}_{\alpha}$-component of $x$.  This map is independent of the choice of the good basis of $M^{\bf T}$. Moreover it is not difficult to see that $\Gamma_{\alpha}$ is a morphism of $\Bbbk {\bf G}_{\alpha}$-modules by Proposition \ref{SL2module}.

Looking at  Example \ref{Example}, we see that the weights of $\mathbb{M}(\theta)_J$, $E(\theta)_J$ and $\nabla(\theta)_J$ have the form of $\{\lambda, \Phi^+_w\}$ for some $\lambda\in \widehat{\bf T}$ and $w\in W$.  This property still holds for general $M\in \mathscr{X}({\bf G})$.

\begin{proposition} \label{wset}
Let $M\in \mathscr{X}({\bf G})$ and $\{\theta, X\}$ be  a weight in $\op{Wt}(M)$. Then there exists an element $w\in W$ such that $X= \Phi^+_w$.
\end{proposition}

\begin{proof}  Let $Y= \Phi^+ \setminus  X$  and set $ \Delta_Y=  \Delta \cap Y$.  We show that $Y=  \Phi^-_w$  by do induction on $|Y|$.
If $|Y|=0$, then the corresponding element $w\in W$ is the identity. If $|Y|=1$,  since $X$ is a self-enclosed  subset, it is not difficult to see that $Y$ consists of a simple root. In this case, we have $Y=\Phi^-_s $, where  $s\in S$ is a simple reflection. Now we suppose that for each $Y$ with $|Y|< m$, there exists an element  $w\in W$ such that $Y=  \Phi^-_w$.

Assume that  $|Y|=m$. Firstly we consider the case that  there is a simple root $\alpha\in \Delta_Y$ such that $\theta|_{{\bf T}_{\alpha}}$ is nontrivial.
According to  Proposition  \ref{SL2module},  $M^{\alpha} $ is a $\Bbbk {\bf G}_{\alpha}$-module.
Thus  the induced $\Bbbk {\bf G}_{\alpha}$-module $\mathbb{M}(\theta^{s_{\alpha}}|_{{\bf T}_{\alpha}}) $ is a direct summand of  $M^{\alpha} $ using Proposition \ref{SL2result}.  Therefore  there exists  $\xi\in M_{\theta}$ with weight $\{\theta, X\}$ such that $ \Bbbk {\bf G}_{\alpha} \xi \cong \mathbb{M}(\theta^{s_{\alpha}}|_{{\bf T}_{\alpha}}) $ as  $\Bbbk {\bf G}_{\alpha}$-modules. Hence the weight of $\dot{s}_{\alpha} \xi$ is
$\{\theta^{s_{\alpha}}, s_{\alpha}(Y \setminus \{\alpha\})\}$. By the inductive assumptions,  there exists  $v\in W$ such that $ s_{\alpha}(Y \setminus \{\alpha\})=\Phi^-_v $, which implies that $Y=\Phi^-_{vs_{\alpha}} $.

Now we consider the another case  that $\theta|_{{\bf T}_{\alpha}}$ is trivial for all $\alpha \in \Delta_Y$. So  for each $\alpha \in \Delta_Y$, there exists an element $\eta_{\alpha}$ of weight $\{\theta, X\}$
such that $\Bbbk {\bf G}_{\alpha} \eta_{\alpha} \cong\mathbb{M}(\op{tr}) $ or $\op{St}$  as  $\Bbbk {\bf G}_{\alpha}$-module.  If there is $\alpha \in \Delta_Y$ such that   $\Bbbk {\bf G}_{\alpha} \eta_{\alpha} \cong\mathbb{M}(\op{tr}) $ as  $\Bbbk {\bf G}_{\alpha}$-module, then we can handle it by the same  discussion as before.
In the following we assume that $\Bbbk {\bf G}_{\alpha} \eta_{\alpha} \cong \op{St}$ as  $\Bbbk {\bf G}_{\alpha}$-module for each  $\alpha \in \Delta_Y$. Thus we have  $s_{\alpha}(Y- \{\alpha\} )=Y- \{\alpha\}$  for each $\alpha\in \Delta_Y$, which also implies that $s_{\alpha}(X)=X$.   In particular, we get $Y \supseteq \Phi_{  \Delta_Y}$. We claim that $Y= \Phi_{  \Delta_Y}$. Suppose that  $Y \supsetneq \Phi_{  \Delta_Y}$ and let $\delta \in Y- \Phi_{  \Delta_Y}$ such that $\text{ht}(\delta)$ is minimal among the roots in $Y- \Phi_{  \Delta_Y}$.
Hence there exists a root $\gamma\in X$ and a simple root $\beta\in \Delta_Y $ such that $\delta=\gamma+ \beta$.
Indeed, suppose that $\delta$ has the form $\delta_1+\delta' +\delta_2 $ with $\delta_1, \delta_2\in \Delta\cap X$ and $\delta'\in \Phi^+$.
Since $\text{ht}(\delta)$ is minimal, we see that $\delta$ is a sum of two roots in $X$ and thus $\delta\in X$, which is a contradiction.
Noting that $\gamma\in X$, we have $s_{\beta}(\gamma) \in X$.
Since $\text{ht}(\delta)$ is minimal and $s_{\beta}(\delta) \geq \delta$,  we get  $s_{\beta}(\gamma) \geq \gamma + 2\beta$.
We consider the element $\zeta: = \dot{s}_{\beta} \varepsilon_{\gamma}(x) \dot{s}_{\beta} \varepsilon_{\beta}(y) \eta_{\beta} $, where $x, y\ne 0$. If ${\bf G}$ is not of type $G_2$, then $s_{\beta}(\gamma) +\beta$ is not a root anymore. Since $ \dot{s}_{\beta} \varepsilon_{\gamma}(x) \dot{s}_{\beta}^{-1} \in {\bf U}_{ s_{\beta}(\gamma) }$ and  $s_{\beta}(\gamma) \in X$,  we have
$$\zeta= \dot{s}_{\beta} \varepsilon_{\gamma}(x) \dot{s}_{\beta}^{-1} \varepsilon_{\beta}(y) \eta_{\beta} = \varepsilon_{\beta}(y) \eta_{\beta}.$$
However noting that $\dot{s}_{\beta}^{-1} \varepsilon_{\beta}(y) \eta_{\beta} \in \Bbbk {\bf U}_{\beta} \eta_{\beta}$, it is not difficult to see that certain element  in ${\bf U}_{\delta}$ will appear in the expression of  $\zeta$  by direct computation. The results we compute by different ways are not the same. Thus we get a contradiction.

When ${\bf G}$ is of type $G_2$, denote by $\alpha_1$ the short  simple root and by $\alpha_2$ the long simple root. Then we have
$$s_{\alpha_1}(\alpha_2)= 3\alpha_1+\alpha_2, \quad  s_{\alpha_2}(\alpha_1)= \alpha_1+\alpha_2. $$
We just need to consider $\Delta_Y=\{\alpha_1\}$ or $\{\alpha_2\}$.  When $ \Delta_Y=\{\alpha_2\}$,  we have $\alpha_1\in X$ and thus $ s_{\alpha_2}(\alpha_1)\in X$. Since $X$ is a self-enclosed subset of $\Phi^+$, it is easy to see that
$X= \Phi^+ -\{\alpha_2\} $.  When $ \Delta_Y=\{\alpha_1\}$,  we have
$$X \supseteq  \{\alpha_2,  3\alpha_1+\alpha_2,    3\alpha_1+2\alpha_2 \}.$$
Note that $s_{\alpha_1}(X)=X$.   If $X$ also contains $\alpha_1+\alpha_2$ or $2 \alpha_1+\alpha_2$, we get $X= \Phi^+ -\{\alpha_1\} $.   Now suppose that
$$Y= \{ \alpha_1, \alpha_1+\alpha_2, 2 \alpha_1+\alpha_2 \}.$$
By  the same discussion as before, we compute the element
$\dot{s}_{\alpha_1} \varepsilon_{\alpha_2}(x) \dot{s}_{\alpha_1} \varepsilon_{\alpha_1}(y) \eta_{\alpha_1}$ by different ways,
where $x, y\ne 0$. The results are not the same  and so we also get a contradiction.  The proposition is proved.
\end{proof}

\section{Abelian category  $\mathscr{X}$ and its simple objects}

In this section, we show that the category  $\mathscr{X}(\bf G)$  introduced in last section is an abelian category.
We also classify all the simple objects in $\mathscr{X}(\bf G)$.  Firstly we introduce and study the highest weight module in $\mathscr{X}(\bf G)$.
According to Proposition \ref{wset}, the weight of  any object in $\mathscr{X}(\bf G)$ has the form $\{\theta, \Phi^+_w\}$, where  $\theta\in \widehat{\bf T}$ and $w\in W$. We write $\{\theta, w\}$ for simplicity. We let
$$ \Omega=\{\{\theta, w\}\mid \theta\in \widehat{\bf T}, w\in W\}$$
be the set of weights. Let $M\in \mathscr{X}(\bf G)$. Thus for any $\xi\in M^{\bf T}$ with weight $\{\theta, w\}$, we have $u\xi=\xi, \forall u\in {\bf U}'_{w}$.

We define an order ``$\prec$" on $ \Omega$.  For $ \{\lambda, w\}$ and $\{\theta, v\}$, we say $ \{\lambda, w\}  \prec \{\theta, v\}$ if there exists $x\in W$ such that $\lambda=\theta^{x}$ and  $x^{-1} ( \Phi^+_w ) \subsetneq \Phi^+_v$.
The second condition is also equivalent to $x( \Phi^-_v ) \subsetneq \Phi^-_w$.
Since  $M\in  \mathscr{X}({\bf G})$ has finitely many weights,  $M$ has highest weights with respect to the order $\prec$.   The corresponding vectors are called highest weight vectors of $M$. If $M$ is generated by one of its  highest weight vector, then we call $M$ a highest weight module in $\mathscr{X}(\bf G)$ .

\begin{proposition} \label{highestweightmodule}
Let $M\in \mathscr{X}({\bf G})$ be a   highest weight module generated by $\xi$ with weight $\{\theta, w\}$. Then there exists a subset $J\subset I$  such that $w=w_J$ and $J\subset I(\theta)$.   Moreover,  any weight  in $\op{Wt}(M)$  has the form $\{\theta^{v}, w_Jv^{-1}\}$ for some $v\in X_J$ and $\dim M_{\{\lambda, v\}}=1$ for any weight $\{\lambda, v\} \in \op{Wt}(M)$. In particular, $\{\theta, w_J\}$ is the unique highest weight of $M$.
\end{proposition}

\begin{proof}
The first conclusion is based on the argument in Proposition \ref{wset}.  If there exists a simple root $\alpha\in \Delta \cap \Phi^-_w$ such that $\Bbbk {\bf G}_{\alpha} \xi \cong  \mathbb{M}(\theta^{s_{\alpha}}|_{{\bf T}_{\alpha}}) $ or $\mathbb{M}(\op{tr})$ as $\Bbbk {\bf G}_{\alpha}$-modules, then we can get a weight higher than $\{\theta, w\}$, which is a contradiction.  Thus  $\theta|_{{\bf T}_{\alpha}}$ is trivial for all $\alpha \in \Delta \cap \Phi^-_w$. According to the proof  in Proposition \ref{wset}, we see that $w=w_J$ and $J\subset I(\theta)$.

Now let $\xi$ be a highest weight vector of weight $\{\theta, {w_J}\}$.  Firstly we  show that $\Bbbk {\bf U} W\xi$ is a submodule of $M$, which implies that $M= \Bbbk {\bf U} W\xi$. Firstly  by Bruhat decomposition, we have $ \Bbbk {\bf G} \xi= \Bbbk {\bf U} W  {\bf U}_{w_J}  \xi$.
Let $\alpha\in J$ and we consider $\dot{s}_{\alpha} \varepsilon_{\alpha}(1) \xi$. Let $\{\xi_1, \xi_2, \dots, \xi_m\}$ be a good basis of $M^{\bf T}$. We assume that  the weight of $\xi_i$ is $\{\theta_i, {w_i}\}$. According to Proposition \ref{SL2module}, we have
 \begin{align} \label{eq3}
\dot{s}_{\alpha} \varepsilon_{\alpha}(1) \xi= \sum_{i=1}^m \sum_{x\in \overline{\mathbb{F}}_q}f_{i,x}  \varepsilon_{\alpha}(x) \xi_i, \quad f_{i,x} \in \Bbbk.
\end{align}
We claim that if  $f_{j,x} \ne 0$ for some $j$,  then $\Phi^+_{w_j} \supseteq  \Phi^+_{w_J}$. Indeed, suppose there exists $j\in \mathbb{N}$ such that $f_{j,x} \ne 0$ and $\Phi^+_{w_j} \supsetneq  \Phi^+_{w_J}$.
Let $\gamma\in \Phi^+_{w_j}\setminus\Phi^+_{w_J} $. Let  $ \varepsilon_{\gamma}(1)$  act on both sides of (\ref{eq3}) and we get a contradiction since $ \varepsilon_{\gamma}(1)$ acts trivially on the  left side of  (\ref{eq3})  but not trivially on the right side of  (\ref{eq3}).
Moreover, it is not difficult to see that when  $f_{j,x} \ne 0$,  then $\theta_j=\theta$. Indeed, firstly we get $\theta_j |_{\op{Ker}\alpha}= \theta |_{\op{Ker} \alpha}$. On the other hand,  $\theta_j |_{{\bf T}_{\alpha}}$ and  $\theta|_{{\bf T}_{\alpha}}$ are both trivial.
Therefore  we get  $\{\theta_j, {w_j}\} \succeq \{\theta, {w_J}\}$ whenever $f_{j,x} \ne 0$. However noting that $\{\theta, {w_J}\}$ is a highest weight, so if $f_{j,x} \ne 0$ for some $x$ then
 $\{\theta_j, {w_j}\}= \{\theta, {w_J}\}$. Let $\eta_1, \eta_2, \dots,  \eta_k\in M^{\bf T}$ such that $\{p_{\theta, w_J}(\eta_1), p_{\theta, w_J}(\eta_2),\dots, p_{\theta, w_J}(\eta_k)\}$ is a basis of $M_{\{\theta, {w_J}\}}$. Then it is easy to  see that $\displaystyle \sum_{j=1}^k\Bbbk {\bf U}_{\alpha}\eta_j$ is a direct sum of Steinberg modules for any $\alpha\in J$.
Therefore $\Bbbk {\bf U}_{\alpha} \xi$ is the Steinberg module for any  $\alpha\in J$. Thus we get
$$ \Bbbk {\bf G} \xi= \Bbbk {\bf U} W  {\bf U}_{w_J}  \xi= \Bbbk {\bf U}  X_J {\bf U}_{w_J}  \xi.$$
Moreover, we have $x{\bf U}_{w_J} x^{-1} \subset {\bf U} $  for $x\in X_J$, and thus $M= \Bbbk {\bf G} \xi= \Bbbk {\bf U}  X_J  \xi.$

Let $L_0= \Bbbk \xi$. For $\alpha\in J$, we have $\dot{s}_{\alpha} \xi =-\xi$.  For $\alpha\in I\setminus J$, if ${\bf U}_{-\alpha} \xi=  \xi$, then we have $  \dot{s}_{\alpha} \xi=\xi $. Otherwise, the weight of  $\dot{s}_{\alpha}\xi$ is $\{\theta^{s_{\alpha}}, w_J s_{\alpha}\}$. Let $L_1= \langle L_0, s_{\alpha} \xi \mid \alpha\in I\rangle$.  We construct $L_i$ by letting $$L_i= \langle L_{i-1}, \dot{w} \xi \mid w\in X_J, \ell(w)=i\rangle.$$
 It is easy to see that there exists a minimal integer $k$ such that $L_{i-1} \subsetneq L_i$ when $i\leq k$ and  $L_{j}=L_{j+1}$ for $j\geq k$.
For $\dot{w}\xi\in L_i/L_{i-1}$, we show that the weight of $\dot{w}\xi$ is $\{\theta^{w}, w_Jw^{-1}\}$ by induction on $i$.  Let $\dot{w}\xi\in L_i/L_{i-1}$, then there exists $v\in X_J$ such that  $w=s_{\beta}v> v$ and $\dot{v}\xi \in L_{i-1}/L_{i-2}$.  If ${\bf U}_{\beta} \dot{w}\xi= \dot{w} \xi$, then ${\bf U}_{-\beta} \dot{v}\xi= \dot{v}\xi$ which implies that  $\dot{s}_{\beta} \dot{v}\xi=\dot{v}\xi$.  We get a contradiction and  thus the weight of $\dot{w}\xi$ is  $\{\theta^{w}, w_Jw^{-1}\}$.  The proposition is proved.
\end{proof}

\begin{remark} \label{successiveproperty}\normalfont
Let $M\in \mathscr{X}({\bf G})$ be a   highest weight module with highest  weight $\{\theta, w_J\}$.
According to the proof of  Proposition \ref{highestweightmodule}, the weights of $M$ has  the successive property:
if $v< sv<w=rsv$ with $r,s\in S$ and $\{\theta^{v}, w_Jv^{-1}\}$, $\{\theta^{w}, w_Jw^{-1}\}$ are weights of $M$, then
$\{\theta^{vs}, w_Jv^{-1}s\}$ is also a weight of $M$.
\end{remark}

Now we show that the category $\mathscr{X}(\bf G)$ is an abelian category. Since the category  $\mathscr{X}(\bf G)$ is a full subcategory of $\Bbbk {\bf G}$-Mod, it is enough to show that  $\mathscr{X}(\bf G)$ is stable under taking subquotients. The following proposition tells us that we just need to show that $\mathscr{X}(\bf G)$ is stable under taking submodules.

\begin{proposition} \label{quotient}
Let $M, N$ be two objects in $\mathscr{X}(\bf G)$ such that  $N$ is  a subobject of $M$, then $M/N$ is also an object in $\mathscr{X}(\bf G)$.
\end{proposition}

\begin{proof}
  Let $\{\xi_1, \xi_2, \dots, \xi_n\}$ be a good basis of  $N^{\bf T}$ and thus $N\displaystyle = \bigoplus_{i=1}^n \Bbbk {\bf U} \xi_i $.  Then  this basis can be extended to a good basis  $\{\xi_1, \xi_2, \dots, \xi_n, \eta_1, \dots, \eta_m\}$ of $M$ by Proposition \ref{constrgoodbasis}.  Thus $M/N \cong  \displaystyle \bigoplus_{i=1}^m \Bbbk {\bf U} \overline{\eta_i}  \in \mathscr{X}(\bf G) $, where $\overline{\eta_i}$ is the quotient of $\eta_i$ in $M/N$.
\end{proof}

\begin{lemma}  \label{5.2}
Let   $M\in \mathscr{X}(\bf G)$ and $N$ be a $\Bbbk {\bf G}$-submodule of $M$. Then $\Bbbk {\bf U} N^T$ is a $\Bbbk {\bf G}$-submodule of $N$. In particular, $\Bbbk {\bf U} N^T$ is an object in $\mathscr{X}(\bf G)$.
\end{lemma}

\begin{proof}  Let $\{\eta_1, \eta_2, \dots,  \eta_n\}$ be a basis of $N^{\bf T}$.  We can extend it to a basis of $M^{T}$ which is denoted by   $\{\eta_1, \eta_2, \dots,  \eta_n, \zeta_1,\zeta_2, \dots, \zeta_m\}$. Suppose that $\Bbbk {\bf U} N^T$ is not a submodule of $N$. Then there exists a simple root $\alpha$ and an integer $k$ ($1\leq k\leq n$) such that
$$\dot{s}_{\alpha} \varepsilon_{\alpha}(1) \eta_k= \sum_{i=1}^n \sum_{x\in \bar{\mathbb{F}}_q } f_{i,x} \varepsilon_{\alpha}(x) \eta_i+ \sum_{j=1}^m \sum_{y\in \bar{\mathbb{F}}_q }  g_{j,y} \varepsilon_{\alpha}(y) \zeta_j \in N$$
and $\zeta: =\displaystyle \sum_{j=1}^m \sum_{y\in \bar{\mathbb{F}}_q }  g_{j,y} \varepsilon_{\alpha}(y) \zeta_j $ is nonzero.
Using Lemma \ref{keylemma}, we have $\displaystyle \sum_{y\in \bar{\mathbb{F}}_q }  g_{j,y}=0$ for any $j=1, 2, \dots, m$.
Indeed, if there exists $j_0$ such that $\displaystyle \sum_{y\in \bar{\mathbb{F}}_q }  g_{j_0,y} \ne 0$, then we can choose a sufficiently large integer $a$ such that $$0 \ne \underline{U_{q^a}} \zeta= \underline{U_{q^a}} (\sum_{j=1}^m g_j \zeta_j) \in N, $$
where $g_j= \displaystyle \sum_{y\in \bar{\mathbb{F}}_q }  g_{j,y}$ for $j=1, 2, \dots, m$.
Thus we have $\displaystyle \sum_{j=1}^m g_j \zeta_j \in N^{\bf T}$ by Lemma \ref{keylemma},  which is a contradiction.  For $M\in \mathscr{X}({\bf G})$, recall  the projection $\Gamma_{\alpha}:  M \rightarrow  M^{\alpha}$.
It is not difficult to see that $\Gamma_{\alpha} N$ is a ${\bf G}_{\alpha}$-submodule $M^{\alpha}$. However noting that $\zeta\in \Gamma_{\alpha} N$, we get a contradiction by Proposition \ref{SL2result}.
\end{proof}

\begin{lemma} \label{5.3}
Let   $M\in \mathscr{X}(\bf G)$ and  $N$ be a $\Bbbk {\bf G}$-submodule of $M$ such that $N^{\bf T}=0$. Then $N=0$.
\end{lemma}

\begin{proof}  Suppose that $N \ne 0$ and let $x\in N$ be a nonzero element. Let $\{\xi_1,\xi_2, \dots, \xi_m\}$ be a good basis of $M$ and suppose the weight of $\xi_i$ is $\{\theta_i, {w_i}\}$ by Proposition \ref{wset} for $i=1,2, \dots, m$. Thus we have $u\xi_i=\xi_i, \forall u\in {\bf U}'_{w_i}$.  For $a\in \mathbb{N}$, let $M_{q^a}= \displaystyle \bigoplus_{i=1}^m \Bbbk U_{q^a} \xi_i$.  Suppose $a$ is big enough such that $x\in M_{q^a}$ and $M_{q^a}$ is a $\Bbbk {\bf G}_{q^a}$-module.  Let $L$ be a simple $\Bbbk {\bf G}_{q^a}$-submodule of $\Bbbk {\bf G}_{q^a}x$.  In the following we show  that $L^{U_{q^a}} \ne 0$.

Since $M_{q^a}$ is a semisimple $\Bbbk {\bf G}_{q^a}$-module and $L$ is a $\Bbbk {\bf G}_{q^a}$-submodule of $M_{q^a}$,  there exists a surjective morphism form $M_{q^a}$ to $L$. We will construct a surjective morphism from $\displaystyle  \bigoplus_{i=1}^m \text{Ind}^{G_{q^a}}_{B_{q^a}} \Bbbk_{\theta_i}$ to $M_{q^a}$. Then by Frobenius reciprocity, we get
$$\text{Hom}_{G_{q^a}} (\bigoplus_{i=1}^m \text{Ind}^{G_{q^a}}_{B_{q^a}} \Bbbk_{\theta_i}, L) \cong \bigoplus_{i=1}^m \text{Hom}_{B_{q^a}}(\Bbbk_{\theta_i}, L|_{B_{q^a}}),$$
  which implies that $L^{U_{q^a}} \ne 0$. Noting that
  $$  \text{Hom}_{G_{q^a}} (\bigoplus_{i=1}^m \text{Ind}^{G_{q^a}}_{B_{q^a}} \Bbbk_{\theta_i}, M_{q^a}) \cong \bigoplus_{i=1}^m \text{Hom}_{B_{q^a}}(\Bbbk_{\theta_i}, M_{q^a}|_{B_{q^a}}), $$
  and $M_{q^a}=\displaystyle \bigoplus_{i=1}^m \Bbbk U_{w_i, q^a} \xi_i$, we will  prove that
  $$ \Psi: \bigoplus_{i=1}^m \text{Ind}^{G_{q^a}}_{B_{q^a}} \Bbbk_{\theta_i} \longrightarrow M_{q^a}: \quad  {\bf 1}_{\theta_i} \mapsto  \underline{ U_{w_i, q^a} } \xi_i, \ \forall i=1,2,\dots, m$$
is a surjective morphism. It is enough to show that one submodule  $V$ of $M_{q^a}$ containing $\underline{ U_{w_1, q^a} } \xi_1, \underline{ U_{w_2, q^a} } \xi_2, \dots, \underline{ U_{w_m, q^a} } \xi_m$ is $M_{q^a}$ itself.

Let $\{\theta, w_J\}$ be a highest weight in $\op{Wt}(M)$ and $\xi_{i_0}$ be an element of weight $\{\theta, w_J\}$.  Then according to the proof of Proposition \ref{highestweightmodule}, $\Bbbk {\bf U}_{\alpha} \xi_{i_0} $ is the Steinberg module of  ${\bf G}_{\alpha}$ for any $\alpha\in J$. Thus $\dot{s}_{\alpha} \xi_{i_0}=-\xi_{i_0}$ and $$\dot{s}_{\alpha}  \varepsilon (x)\xi_{i_0}=(\varepsilon (-\frac{1}{x})-1 ) \xi_{i_0}.$$
Therefore $\Bbbk U_{w_J, q^a}\xi_{i_0}$ is the Steinberg module of $G_{J, q^a}$. It is well know that this module is a simple  $\Bbbk G_{J, q^a}$-module
when $\text{char}\ \Bbbk =0$. Noting that $\underline{U_{w_J, q^a}}\xi_{i_0} \in V$,  we get $\xi_{i_0} \in V$, which implies that
$\Bbbk W \xi_{i_0} \in V$.  If $\Bbbk W \xi_{i_0} = M^{\bf T}$, then we have proved $V=M_{q^a}$.
By the proof of Proposition \ref{highestweightmodule}, we see that  $\Bbbk {\bf U} W\xi_{i_0}$ (resp. $\Bbbk {\bf U}_{q^a} W\xi_{i_0}$) is a submodule of $M$ (resp. $M_{q^a}$).  According to Proposition \ref{quotient}, $M/\Bbbk {\bf U} W\xi_{i_0}$ is an object in $ \mathscr{X}(\bf G)$.
Let $\{\lambda, K\}$ be a highest weight of $M/\Bbbk {\bf U} W\xi_{i_0}$. Thus $\{\lambda, K\}$ is also a weight of $M$.
Then there exists an element  $\xi_{_{i_1}}$ such that $\overline{\xi}_{i_1}$ is the  corresponding highest weight vector of  $M/\Bbbk {\bf U} W\xi_{i_0}$. Thus we get $\overline{\xi}_{i_1} \in V/ \Bbbk {\bf U}_{q^a} W\xi_{i_0}$ by the same discussion as before. Hence we get $\xi_{i_1}\in V$.
So $\Bbbk W \xi_{i_0} + \Bbbk W \xi_{i_1} \in V$.  If $\Bbbk W \xi_{i_0} + \Bbbk W \xi_{i_1}= M^{\bf T}$, then we get  $V=M_{q^a}$.
Otherwise, we can have the similar discussion step by step and finally we get $M^{T} \subseteq V$, which implies that $V=M_{q^a}$.

Now we have $(\Bbbk {\bf G}_{q^a}x)^{U_{q^a}} \ne 0$. Noting that $\displaystyle (M_{q^a})^{U_{q^a}}= \bigoplus_{i=1}^m   \Bbbk \underline{ U_{w_i, q^a} } \xi_i,$ there exists an element
$$\zeta= \sum_{i=1}^m   f_i \underline{ U_{w_i, q^a} } \xi_i \in \Bbbk {\bf G}_{q^a}x \cap  (M_{q^a})^{U_{q^a}},$$
where $f_i\in \Bbbk$ for $i=1,2,\dots, m$. Therefore we get
$$ \underline{U_{q^a}} \zeta =\underline{U_{q^a}}  \sum_{i=1}^m f_iq^{a\ell(w_i)}\xi_i \in \Bbbk {\bf G}_{q^a}x\subseteq N.$$
Following Lemma \ref{keylemma}, we have $\displaystyle \sum_{i=1}^m f_iq^{a\ell(w_i)}\xi_i \in N$.  This element is in $N^{\bf T}$ and thus we get a contradiction. The lemma is proved.

\end{proof}

\begin{theorem}
The category $\mathscr{X}(\bf G)$ is an abelian category.
\end{theorem}

\begin{proof}
  By Proposition \ref{quotient}, it is enough to show that any submodule of $M\in \mathscr{X}(\bf G)$ is also an object in $\mathscr{X}(\bf G)$.
Let $N$ be a submodule of $M$ and set $N_1= \Bbbk {\bf U} N^{\bf T}$.  According to Lemma \ref{5.2}, we see that $N_1\in \mathscr{X}(\bf G)$.
Since $N\subseteq M$, it is easy to see that $(N/N_{1})^{\bf T} =0$. Noting that $M/N_{1} \in \mathscr{X}(\bf G) $ and using Lemma \ref{5.3},   we have $N_1=N$, which implies that $N\in \mathscr{X}(\bf G)$.

\end{proof}

\begin{proposition}
  Given a short exact sequence
$$0 \longrightarrow M_1  \longrightarrow M  \longrightarrow M_2  \longrightarrow 0$$
in the category $\mathscr{X}(\bf G)$, one has that $M^{\bf T} \cong M^{\bf T}_1 \oplus M^{\bf T}_2$ as ${\bf T}$-modules and
$\op{ch} M= \op{ch} M_1+ \op{ch} M_2 $. In particular, we have $\op{Wt}(M)=\op{Wt}(M_1) \cup \op{Wt}(M_2)$.
\end{proposition}

\begin{proof} According to  \cite[Section 2]{FS},  $\Bbbk {\bf B}$ is a locally Wedderburn algebra when $\Bbbk$ is an algebraically closed field of characteristic zero. Thus by \cite[Lemma 3]{FS}, $\Bbbk_{\theta}$ is an injective $\Bbbk {\bf B}$-module which implies the exactness of $\text{Hom}_{\bf B}(-, \Bbbk_{\theta})$. It is easy to see that  for $M\in \mathscr{X}(\bf G)$, we have
$$ \text{Hom}_{\bf T} (\Bbbk_{\theta}, M) \cong \text{Hom}_{\bf B} (M, \Bbbk_{\theta}),$$
which implies the exactness of $\text{Hom}_{\bf T}(\Bbbk_{\theta}, -)$. Thus  we get $M^{\bf T} \cong M^{\bf T}_1 \oplus M^{\bf T}_2$ as ${\bf T}$-modules.

Let $B_1$ be a good basis of $M_1$. We can extend it to a good basis $B_1 \cup B_2$ of $M$ by Proposition \ref{constrgoodbasis}.
Thus $\overline{B_2}$ (the quotient of $B_2$ in $M/M_1$) is a good basis of $M_2$.
Hence  $\text{ch} M= \text{ch} M_1+ \text{ch} M_2 $.

\end{proof}

\begin{corollary} \label{finitelength}
The category  $\mathscr{X}(\bf G)$ is noetherian and artinian. In particular, each object in $\mathscr{X}(\bf G)$ is of finite length.
\end{corollary}

Let $\Omega_0=\{\{\theta, {w_J}\} \mid \theta\in \widehat{\bf T}, J\subset I(\theta)\}$. Then by Proposition \ref{highestweightmodule}, the highest weights of an object $M\in \mathscr{X}(\bf G)$ are in $\Omega_0$. By definition, all simple objects in $\mathscr{X}({\bf G})$ are highest weight modules.

\begin{lemma}
Let $M, N$ be two simple objects in  $\mathscr{X}({\bf G})$ with the same highest weight $\{\theta, {w_J}\} \in \Omega_0$. Then $M\cong N$.
\end{lemma}

\begin{proof} Let $\xi$ be the highest weight vector of $M$  and $\eta$  be the highest weight vector of  $N$.
Then $\xi+\eta$ is a highest weight vector of $M \oplus N$.  Let $V$ be the submodule of $M \oplus N$  generated by $\xi+\eta$.
Noting that  $\dim V_{\{\theta, {w_J}\}}=1$,  we see that  $V\ne M \oplus N $. Since $ M \oplus N $ has only two composition factors,   $V$
is a simple object in $\mathscr{X}({\bf G})$. We let $p_M:  M \oplus N \rightarrow M $ and $p_N:  M \oplus N \rightarrow N$.  Then $p_M, p_N$ are $\Bbbk {\bf G}$-module morphism. Since $p_M|_V$ and $p_N|_V$ are nonzero, we get $M \cong  N \cong (M \oplus N)/ V$.
\end{proof}

The $\Bbbk {\bf G}$-modules $E(\theta)_J$ has highest weight $\{\theta, {w_J}\}\in \Omega_0$. By  \cite[Proposition 2.8]{CD3},   $E(\theta_1)_{K_1}$ is isomorphic to $E(\theta_2)_{K_2}$ as $\Bbbk {\bf G}$-modules if and only if $\theta_1=\theta_2$ and $K_1=K_2$.
Thus we have classified  all the simple objects in $\mathscr{X}({\bf G})$ by the above lemma.

\begin{theorem} \normalfont
The  $\Bbbk {\bf G}$-modules $E(\theta)_J$  ($J\subset I(\theta)$) are all the simple objects in $\mathscr{X}({\bf G})$.  Moreover,  we have a bijection:
$$ \Omega_0 \longrightarrow \op{Irr} (\mathscr{X}({\bf G})):  \quad (\theta, w_J) \mapsto E(\theta)_J,$$
where $\op{Irr} (\mathscr{X}({\bf G}))$ is the set of the simple objects in $\mathscr{X}({\bf G})$.

\end{theorem}

According to Example \ref{Example} (b), it is easy to see that two simple objects $M$ and $N$ are isomorphic if and only if $\op{ch} M= \op{ch} N$.

\section{Highest weight category}

In this section, we show that $\mathscr{X}(\bf G)$ is a highest weight category.  Firstly we recall the definition of highest weight category (see \cite{CPS}).

\begin{definition}\label{HWC} \normalfont

Let $\mathscr{C}$ be a locally artinian, abelian, $\Bbbk$-linear category with enough injective objects that satisfies Grothendieck's condition. Then we call $\mathscr{C}$ a highest weight category if there exists a locally finite poset $\Lambda$ (the ``weights" of $\mathscr{C}$), such that:

\noindent (a) There is a complete collection $\{S(\lambda)_{\lambda \in \Lambda}\}$ of non-isomorphic simple objects of $\mathscr{C}$ indexed by the set $\Lambda$.

\noindent  (b) There is a collection $\{A(\lambda)_{\lambda \in \Lambda}\}$ of objects of $\mathscr{C}$ and, for each $\lambda$, an embedding $S(\lambda)\subset A(\lambda) $ such that all composition factors $S(\mu)$ of $A(\lambda)/S(\lambda)$ satisfy $\mu < \lambda$. For $\lambda,\mu \in \Lambda$,
we have that $\dim_{\Bbbk}\op{Hom}_{\mathscr{C}}(A(\lambda), A(\mu))$ and $[A(\lambda): S(\mu)]$ are finite.

\noindent  (c) Each simple object $S(\lambda)$ has an injective envelope $\mathcal{I}(\lambda)$ in $\mathscr{C}$.
Also, $\mathcal{I}(\lambda)$  has a good filtration $0= F_0(\lambda)\subset F_1(\lambda)\subset \dots $ such that:

\noindent (i) $F_1(\lambda)\cong A(\lambda)$;

\noindent  (ii) for $n>1$, $F_n(\lambda)/F_{n-1}(\lambda) \cong A(\mu)$ for some $\mu=\mu(n)> \lambda$;

\noindent  (iii) for a given $\mu \in \Lambda$, $\mu=\mu(n)$ for only finitely many $n$;

\noindent  (iv) $\bigcup F_i(\lambda)= \mathcal{I}(\lambda)$.
\end{definition}

The main  difficulty in proving that $\mathscr{X}(\bf G)$ is a highest weight category  is  to prove that the category $\mathscr{X}({\bf G})$ has enough injective objects. In the following we will show that  $\nabla(\theta)_J$ is an injective object in  category $\mathscr{X}(\bf G)$. For  $J\subset I(\theta)$, the simple object $E(\theta)_J$ is the socle of  $\nabla(\theta)_J$ by Proposition \ref{nablaProp}. Therefore  $\mathscr{X}(\bf G)$ has enough injective objects.

\begin{proposition}  \label{zeroExt1}
 Let $N\in \mathscr{X}(\bf G)$ such that the weights of $N$ are not less than or equal to  $\{\theta, w_J\} \in \Omega_0$.  Then $\op{Ext}^1_{\mathscr{X}(\bf G)}(E(\theta)_J,N)=0$.
\end{proposition}

\begin{proof}  It is enough to show that
$\text{Ext}^1_{\mathscr{X}(\bf G)}(E(\theta)_J,N)=0$ for any simple object $N \in \mathscr{X}(\bf G)$ whose weights are not less than or equal to  $\{\theta, w_J\} \in \Omega_0$.
Let
\begin{align} \label{SES1}
0 \longrightarrow N   \longrightarrow L \longrightarrow E(\theta)_J \longrightarrow 0
\end{align}
be a short exact sequence in  $\mathscr{X}(\bf G)$.  Since $\text{ch} L=\text{ch} E(\theta)_J+ \text{ch} N$ and the highest weight of $E(\theta)_J$ is $\{\theta, w_J\}$, we let $\xi \in L$ be a vector of weight $\{\theta, w_J\}$. Note that  $\{\theta, w_J\}$ is also a highest weight of $L$.  The submodule $\Bbbk {\bf G}\xi$ of $L$ has weights  less than
or equal to $\{\theta, w_J\}$ by Proposition \ref{highestweightmodule}. Thus  $\Bbbk {\bf G}\xi\cong E(\theta)_J$. So the short exact sequence (\ref{SES1}) is splitting and thus $\text{Ext}^1_{\mathscr{X}(\bf G)}( E(\theta)_J,N)=0$.
\end{proof}

\begin{lemma} \label{zeroExt2}
Let $\lambda, \mu \in\widehat{{\bf T}}$ be two different characters.  Then $$\op{Ext}^1_{\mathscr{X}(\bf G)}(E(\theta)_J, E(\lambda)_K)=0,$$
where $J\subset I(\theta)$ and $K\subset I(\lambda)$. In particular, for any  $M, N$  in $\mathscr{X}(\bf G)$ such that $\op{Ev}(M) \cap \op{Ev}(N)=\emptyset$, we have $\op{Ext}^1_{\mathscr{X}(\bf G)}(M,N)=0$.
\end{lemma}

\begin{proof} Suppose $\op{Ext}^1_{\mathscr{X}(\bf G)}(E(\theta)_J, E(\lambda)_K)\ne 0$ and let
\begin{align} \label{SES2}
0 \longrightarrow   E(\lambda)_K    \longrightarrow M  \longrightarrow E(\theta)_J  \longrightarrow 0
\end{align}
be a non-splitting short exact sequence. Let $\xi$ be a weight vector of $M$ with weight $\{\theta, w_J\}$.
Thus   $M= \Bbbk {\bf G}\xi$ is a highest weight module of weight  $\{\theta, w_J\}$.
 Then $\{\lambda, w_K\} = \{\theta^v, w_Jv^{-1}\}$ for some $v\in X_J$ by Proposition \ref{highestweightmodule}.
However noting that $v\in W_K$ and $K \subset I(\lambda)$, then we get $\lambda=\theta$, which is a contradiction.
\end{proof}

\begin{lemma}  \label{zeroExt3}
For $\{\theta, w_J\} \in \Omega_0$,  we have $\op{Ext}^1_{\mathscr{X}(\bf G)}(E(\theta)_J,E(\theta)_J)=0$.
\end{lemma}

\begin{proof}
We let
\begin{align} \label{SES3}
0 \longrightarrow  E(\theta)_J   \longrightarrow M \longrightarrow E(\theta)_J  \longrightarrow 0
\end{align}
be a  short exact sequence. Then $\dim M_{\{\theta, w_J\}} =2$. Let $\xi, \eta \in M^{\bf T}$ such that $p_{\theta, w_J}(\xi), p_{\theta, w_J}(\eta)$ are linearly independent.
Using Proposition \ref{highestweightmodule},    we see that $\Bbbk {\bf G} \xi $ and  $\Bbbk {\bf G} \eta$ are highest weight modules and thus they are  both isomorphic to $E(\theta)_J$. So the short exact sequence (\ref{SES3}) is splitting and thus the lemma is proved.
\end{proof}

\begin{theorem} \label{injectivemodule}
For $J\subset I(\theta)$, $\nabla(\theta)_J$ is an injective object in  the category $\mathscr{X}(\bf G)$.
\end{theorem}

\begin{proof} It is enough to show that $\text{Ext}^1_{\mathscr{X}(\bf G)}(\mathcal{S},\nabla(\theta)_J)=0$  for any simple object $\mathcal{S} \in \mathscr{X}(\bf G)$.
 Since $\nabla(\theta)_J$ has the composition factors $E(\theta)_L$ ( $L \subset J$) with multiplicity one, we just need to consider $\mathcal{S}=E(\theta)_K$
 for $K \subset J$ by Proposition \ref{zeroExt1} and Lemma \ref{zeroExt2}.
 Let $M$  be the  submodule of $\nabla(\theta)_J$  whose composition factors are $E(\theta)_X$ with $ K \subset X \subset J$.  Set $N= \nabla(\theta)_J/M$.
According to Lemma \ref{zeroExt2}, we get  $\text{Ext}^1_{\mathscr{X}(\bf G)}(E(\theta)_K, N)=0$. In the following we show that $\text{Ext}^1_{\mathscr{X}(\bf G)}(E(\theta)_K, M)=0$.

When $K=J$, we have $M= E(\theta)_K$ and thus $\text{Ext}^1_{\mathscr{X}(\bf G)}(E(\theta)_K, M)=0$ by Lemma \ref{zeroExt3}.
Now we assume that  $K \subsetneq J$ and let
\begin{align} \label{SES4}
0 \longrightarrow  M    \longrightarrow M' \longrightarrow E(\theta)_K  \longrightarrow 0
\end{align}
be a short exact sequence.  Note that $M$ is a highest weight module generated by a highest weight vector of weight  $\{\theta, w_K\}$. Since $\dim M'_{\{\theta, w_K\}}=2$,  we let $\xi, \eta \in (M')^{\bf T}$
such that $p_{\theta, w_J}(\xi), p_{\theta, w_J}(\eta)$ are linearly independent.
 Without loss of generality, we can assume that  $\Bbbk {\bf G} \xi\cong M$. Thus it is easy to see that $E(\theta)_K$ is  a simple quotient of the $\Bbbk {\bf G}$-modules $\Bbbk{\bf G} \eta$.
Let $B_1= \displaystyle \bigcup_{ K \subset X \subset J } \bigcup_{w\in Z_X(\theta)}\{ \dot{w}\xi \}$, which is a good basis of $M$. Set $B_2= \bigcup_{w\in Z_K(\theta)}\{ \dot{w}\eta \}$. Then we see that
$B_1 \cup B_2$ is a good basis of $M'$ by Proposition \ref{constrgoodbasis}. Now assume  that $\Bbbk {\bf G} \eta$ is not isomorphic to $E(\theta)_K$. We will show that $\Bbbk {\bf G} \eta$ is isomorphic to $M$. Suppose that $\Bbbk {\bf G} \eta$ is not isomorphic to $M$, then according to Proposition \ref{highestweightmodule} and Remark \ref{successiveproperty}, there exists a vector $\dot{v}\eta\in \Bbbk{\bf G} \eta$ of weight $\{\theta^{v}, w_Kv^{-1}\}$
such that $v\notin Z_K(\theta)$   and $\{\theta^{sv}, w_Kv^{-1}s\}$ is not a weight of $\Bbbk {\bf G} \eta$ anymore but still a weight of $M$. Noting that  $B_1 \cup B_2$ is a good basis of $M'$,  we have
\begin{align} \label{eq4}
\dot{v}\eta= \sum_{\dot{w}\xi\in B_1} a_w \dot{w}\xi +  \sum_{\dot{w}\eta\in B_2} b_w \dot{w}\eta.
\end{align}
Since $\displaystyle \dim (M')_{\{\theta^{sv}, w_Kv^{-1}s \}}=\dim M_{\{\theta^{sv}, w_Kv^{-1}s \}}=1$, we see that $a_v\ne 0$.
Moreover if $a_w\ne 0$, we get ${\bf U}_{w_Kw^{-1}} \subseteq  {\bf U}_{w_Kv^{-1}}$ by Proposition \ref{expressionofbasis}. Let $\dot{s}$ act  on both sides of (\ref{eq4}).  Since $\{\theta^{sv}, w_Kv^{-1}s\}$ is not a weight of $\Bbbk {\bf G} \eta$  but still  a weight of $\Bbbk {\bf G} \xi$, it is not difficult to see that $\dot{s}\dot{v}\xi$ is a linearly combination of elements in $(B_1 \setminus  \{\dot{s}\dot{v}\xi\}) \cup B_2 $, which is a contradiction.
Now we have proved that $\Bbbk {\bf G} \eta$ is isomorphic to $M$ when $\Bbbk {\bf G} \eta$ is not isomorphic to $E(\theta)_K$.

If $\Bbbk {\bf G} \eta$ is isomorphic to $E(\theta)_K$, then the short exact sequence (\ref{SES4})  is splitting  and the theorem is proved. Otherwise, $\Bbbk {\bf G} \eta$ is isomorphic to $M$ according to the previous discussion. For  $L$ with $K\subsetneq L \subseteq J$,  let $\dot{w}\xi$ and $\dot{w}\eta$ be the elements whose weights are $\{\theta, w_L\}$.  However noting that $\dim (M')_{\theta, w_L}=1$, there exists $f,g\in \Bbbk^*$ such that $\dot{w}(f\xi+g\eta)$ is not of weight $\{\theta, w_L\}$.  Set $\zeta=f\xi+g\eta$. Hence $\Bbbk {\bf G} \zeta$ is not  isomorphic to $M$. Thus we get $\Bbbk {\bf G} \zeta \cong E(\theta)_K$ and the theorem is proved.

\end{proof}

Now we can prove  that $\mathscr{X}({\bf G})$ is a highest weight category. In Definition \ref{HWC}, the set of weights is $ \Omega_0=\{(\theta ,w_J) \mid \theta\in \widehat{{\bf T}}, J\subset I(\theta)\}$.  The order  of  $\Omega_0$ is defined as follows:
$$(\theta_1, w_{J_1}) <  (\theta_2, w_{J_2} ), \ \text{if}\ \theta_1=\theta_2 \ \text{and} \ J_1\supsetneq J_2.$$
We let $S(\{\theta ,w_J\})= A(\{\theta ,w_J\})=E(\theta)_J$ and $\mathcal{I}(\{\theta ,w_J\})= \nabla(\theta)_J$.
With these settings,  it is not difficult to check that $\mathscr{X}({\bf G})$  satisfies  all the conditions in Definition \ref{HWC} and thus we get  the following theorem.

\begin{theorem} \label{XGisHWC}
The category $\mathscr{X}({\bf G})$  is a highest weight category.
\end{theorem}

Some discussions and results in \cite[Section 4]{D1} still hold in category $\mathscr{X}({\bf G})$. We list some interesting properties for the
completeness. One can refer to \cite[Section 4 and Section 5]{D1} for more details.

\begin{proposition}\label{enough projectives}
For $\theta\in \widehat{\bf T}$ and $J\subset I(\theta)$,  $\mathbb{M}(\theta)_J$ is the projective cover of $E(\theta)_J$.
\end{proposition}

\begin{proposition}\label{Extension}

\noindent $(1)$  For $n\geq 0$, $\op{Ext}^n_{\mathscr{X}({\bf G})}(M,N)$ is finite-dimensional for all $M,N\in \mathscr{X}({\bf G})$.

\noindent $(2)$ If $\op{Ext}^n_{\mathscr{X}({\bf G})}(E(\lambda)_J,E(\mu)_K)\ne 0$, then $\lambda=\mu$ and $J\subseteq K$. Moreover,  if $n>0$, we have   $\lambda=\mu$ and $J\subsetneq K$.
\end{proposition}

For a fixed $\theta\in \widehat{\bf T}$, let $\mathscr{X}({\bf G})_\theta$ be the subcategory of $\mathscr{X}({\bf G})$ containing the objects whose subquotients are $E(\theta)_J$ for some $J\subset I(\theta)$.
Then by Proposition \ref{Extension}, we have $\mathscr{X}({\bf G})= \displaystyle \bigoplus_{\theta\in \widehat{\bf T}} \mathscr{X}({\bf G})_\theta$. For each $\theta\in \widehat{\bf T}$, $\mathscr{X}({\bf G})_\theta$ is a highest weight category and then there exists a finite-dimensional quasi-hereditary algebra $A_\theta$ such that $\mathscr{X}({\bf G})_\theta$ is equivalent to the right $A_\theta$-modules.
Indeed, if we set $ \mathscr{I}_{\theta}=\displaystyle\bigoplus_{J\subset I(\theta)}\nabla(\theta)_J$, then $A_\theta\cong \op{End}_{\mathscr{X}({\bf G})}(\mathscr{I}_{\theta})$. The functor $\op{Hom}_{\bf G}(-, \mathscr{I}_{\theta})^*$ form $\mathscr{X}({\bf G})_\theta$ to the right $A_\theta$-modules is an equivalence of categories. Therefore we also see that  the category $\mathscr{X}({\bf G})$ is a Krull-Schmidt category.

Assume that $X$ is a finite set with  cardinality $|X|=n\geq 1$.
Denote by $M_{2^n}(\Bbbk)$ the matrix algebra over $\Bbbk$. The rows and columns of a matrix in this algebra are indexed by the subsets of $X$.
Fix an order of the subsets of $X$, we let $$\mathscr{A}_n= \{ (a_{Y,Z})\in M_{2^n}(\Bbbk) \mid a_{Y,Z}=0 \ \text{if}~Y ~\text{is not a subset of } Z \}$$
which is a subalgebra of the matrix algebra $ M_{2^n}(\Bbbk)$.  The algebra $\mathscr{A}_n$ just depends on the cardinality of $X$.
Then $A_\theta \cong \mathscr{A}_n $ as $\Bbbk$-algebras where  $n=|I(\theta)|$. Thus for $\lambda, \mu\in \widehat{\bf T}$,  if $|I(\lambda)|=|I(\mu)|$, then $A_\lambda \cong A_\mu$ as $\Bbbk$-algebras and thus $\mathscr{X}({\bf G})_\lambda$ is equivalent to  $\mathscr{X}({\bf G})_\mu$.

\begin{proposition} \cite[Proposition 5.1]{D1} \label{typealgebra}
The  algebras $\mathscr{A}_1,  \mathscr{A}_2$ are of finite type and $\mathscr{A}_3$ is of tame type.  For $n\geq 4$,  $\mathscr{A}_n$ is  of wild type.
\end{proposition}

\bigskip

\noindent{\bf Acknowledgements} The author is grateful to  Nanhua Xi,  Xiaoyu Chen, Zongzhu Lin and Tao Gui  for their suggestions and helpful discussions. The work is sponsored by Shanghai Sailing Program (No.21YF1429000) and NSFC-12101405.

\medskip

\noindent{\bf Statements and Declarations}  The author declares that he has no conflict of interests with others.

\bigskip

\bibliographystyle{amsplain}

\end{document}